\documentclass[11pt]{article}

\usepackage{amsmath,amsfonts,amssymb} 

\usepackage{amsthm} 

\usepackage{fullpage} 
\usepackage[colorlinks]{hyperref} 
\usepackage[parfill]{parskip} 

\allowdisplaybreaks 

\begingroup
    \makeatletter
    \@for\theoremstyle:=definition,remark,plain\do{%
        \expandafter\g@addto@macro\csname th@\theoremstyle\endcsname{%
            \addtolength\thm@preskip\parskip
            }%
        }
\endgroup

\newcommand{\mc}{\mathcal}

\newtheorem{theorem}{Theorem}[section]

\newtheorem{lemma}[theorem]{Lemma}
\newtheorem{conjecture}[theorem]{Conjecture}
\newtheorem{question}[theorem]{Question}
\newtheorem{defn}[theorem]{Definition}

\newtheorem{claim}[theorem]{Claim}

\title{Monochromatic paths and cycles in $2$-edge-colored graphs with large minimum degree}
\date{\today}
\author{
J\' ozsef Balogh~\thanks{Department of Mathematics, University of Illinois at Urbana--Champaign, IL, 
USA, and Moscow Institute of Physics and Technology, Russian Federation, 
jobal@illinois.edu. 
Research of this author is partially supported by  NSF Grant  DMS-1764123, Arnold O. Beckman Research Award (UIUC) Campus Research Board 18132 and the Langan Scholar Fund (UIUC).}
\and Alexandr Kostochka \thanks{Department of Mathematics, University of Illinois at Urbana--Champaign, IL, USA and
Sobolev Institute of Mathematics, Novosibirsk 630090, Russia, kostochk@math.uiuc.edu. Research of this author is supported in part by NSF grant
 DMS-1600592, Arnold O. Beckman Research Award (UIUC) RB20003  and by grants 18-01-00353  and 19-01-00682 of the Russian Foundation for Basic Research.}
  \and Mikhail Lavrov\thanks{Department of Mathematics, Kennesaw State University, Marietta, GA, mlavrov@kennesaw.edu; the work was partially done while M. Lavrov was a postdoc at Department of Mathematics, University of Illinois at Urbana–-Champaign.}
 \and Xujun Liu\thanks{Coordinated Science Laboratory, University of Illinois at Urbana--Champaign, IL, USA, xliu150@illinois.edu; the work was partially done while X. Liu was a PhD student at Department of Mathematics, University of Illinois at Urbana--Champaign. Research of this author was supported by the Waldemar J., Barbara G., and Juliette Alexandra Trjitzinsky Fellowship.
}
 }

\date{\it Dedicated to the memory of Richard H. Schelp}
\begin{document}
	\maketitle

\begin{abstract}
A graph $G$ {\em arrows} a graph $H$  if in every $2$-edge-coloring of $G$ there exists a monochromatic copy of $H$.
Schelp had the idea that if the complete graph $K_n$ arrows a small graph $H$, then every ``dense" subgraph of $K_n$ also arrows $H$, and he outlined some problems in this direction. Our main result is in this spirit. We prove that for every sufficiently large $n$, if
$n = 3t+r$ where $r \in \{0,1,2\}$ and $G$ is an $n$-vertex graph with $\delta(G) \ge (3n-1)/4$, then for every  $2$-edge-coloring of $G$,  either there are cycles of every length $\{3, 4, 5, \dots, 2t+r\}$ of the same  color, or 
there are cycles of every even length $\{4, 6, 8, \dots, 2t+2\}$ of the same color.

Our result is tight in the sense that no longer cycles (of length $>2t+r$) can be guaranteed and the minimum degree condition cannot be reduced. It also implies the conjecture of Schelp  that for every sufficiently large $n$,  every $(3t-1)$-vertex graph $G$ with minimum degree larger than $3|V(G)|/4$ arrows the path $P_{2n}$ with $2n$ vertices.
Moreover, it implies for sufficiently large $n$  the conjecture by Benevides, \L uczak, Scott, Skokan and White that for $n=3t+r$ where $r \in \{0,1,2\}$ and every $n$-vertex graph $G$ with $\delta(G) \ge 3n/4$, in each $2$-edge-coloring of $G$ there exists a monochromatic cycle of length at least $2t+r$. 
\\
\\
 {\small{\em Mathematics Subject Classification}: 05C15, 05C35, 05C38.}\\
 {\small{\em Key words and phrases}:  Ramsey number,  Regularity Lemma, paths and cycles.}
\end{abstract}

\section{Introduction}

We consider Ramsey properties of paths and cycles for $2$-edge-colorings of graphs. 
For graphs $G_0,G_1,G_2$ we write $G_0 \mapsto (G_1,G_2)$ if for every $2$-edge-coloring of $G_0$, for some $i\in [2]$ there is a copy of $G_i$ with all edges of color $i$. If $G_0 \mapsto (G_1,G_1)$, we  say that $G_0$ {\em arrows} $G_1$.
The {\em Ramsey number} $R(G_1,G_2)$ is the minimum $N$ such that $K_N\mapsto (G_1, G_2)$.

We  denote by $P_k$ the path with $k$ vertices, and
by $C_k$ the  cycle with $k$ vertices. 
The  {\em circumference}, $c(G)$, of a graph $G$ is the length of a longest cycle in $G$. A graph $G$ on $n$ vertices is {\em pancyclic} if it contains cycles of every integer length in $[3,n]$. Let $E(G) = E(R) \cup E(B)$ be a $2$-edge-coloring of $G$; {\em the monochromatic circumference} is the length of a longest monochromatic cycle in $G$. A balanced bipartite graph $G$ on $2n$ vertices is {\em bipancyclic} if it contains cycles of every even length in $[4,2n]$.

The study of Ramsey-type problems of paths started from the seminal paper by Gerencs\' er and Gy\' arf\' as~\cite{GG1}. They proved that {\em for any choice of positive integers $k\geq \ell$,  $R(P_{k},P_{\ell})=k-1+\left\lfloor \frac{\ell}{2}\right\rfloor$,} which implies that $K_{3n-1}$ arrows $P_{2n}$.

Later, there was a series of papers proving that some dense subgraphs of $K_{3n-1}$ also arrow $P_{2n}$. In particular,  Gy\' arf\' as,  Ruszink\' o,   S\' ark\"{o}zy and  Szemer\' edi~\cite{GRSS0} showed that
{\em  $K_{n,n,n}\mapsto (P_{2n-o(n)},P_{2n-o(n)})$.}
Their conjecture that {\em $K_{n,n,n}\mapsto (P_{2n+1},P_{2n+1})$ } was recently proved for  large $n$ in~\cite{BKLL2}.

More generally, Schelp had the idea that for many graphs $H$, if $K_n$ arrows $H$ then each ``sufficiently dense" subgraph of $K_n$ also arrows $H$. In~\cite{schelp12} he discussed some specific graphs $H$, and different notions of density. One natural measure of density is the minimum degree.
Schelp asked some questions and outlined possible directions of study of this phenomenon. 

Li, Nikiforov and Schelp~\cite{LNS} conjectured the following and proved a partial result.

\begin{conjecture}[\cite{LNS}]\label{LNS-conj}
Let $G$ be a graph of order $n \ge 4$ with $\delta(G)>3n/4$ and let $E(G) = E(R) \cup E(B)$ be a $2$-edge-coloring of $G$. For all $\ell \in [4, \lceil n/2 \rceil]$, $G$ arrows $C_{\ell}$.
\end{conjecture}

\begin{theorem}[\cite{LNS}]\label{12}
Let $\epsilon > 0$ and $G$ be a graph of sufficiently large order $n$ with minimum degree $\delta(G)>3n/4$. Let $E(G) = E(R) \cup E(B)$ be a $2$-edge-coloring of $G$. For all $\ell \in [4, (1/8 - \epsilon)n]$, $G$ arrows $C_{\ell}$.
\end{theorem}


Moreover, having in mind that $R(P_{2n},P_{2n})=3n-1$, Schelp~\cite{schelp12} posed the following conjecture.

\begin{conjecture}[\cite{schelp12}]\label{schelp}
Suppose that $n$ is large enough and $G$ is a graph on $3n-1$ vertices with minimum degree larger than $3|V(G)|/4$. 
Then $G$ arrows $P_{2n}$.
\end{conjecture}

Gy\'arf\'as and S\'ark\"ozy~\cite{gyarfas12} and independently Benevides, \L uczak, Scott, Skokan and White~\cite{BLSSW}
proved an asymptotic version of Conjecture~\ref{schelp}.
In fact,  Benevides, \L uczak, Scott, Skokan and White~\cite{BLSSW} proved a stronger result:

\begin{theorem}[Theorem 1.8 in \cite{BLSSW}]
\label{benevides}
For every $0<\delta\leq 1/180$, there exists an integer $n_0=n_0(\delta)$ such that the following holds.
Let $G$ be a graph of order $n>n_0$ with $\delta(G)\geq 3n/4$. Suppose that $E(G)=E(R_G)\cup E(B_G)$ is 
a $2$-edge-coloring of $G$.  Then either $G$ has monochromatic circumference at least $(2/3+\delta/2)n$ or 
one of $R_G$ and $B_G$ contains cycles of all lengths $\ell\in[3,(2/3-\delta)n]$.
\end{theorem}

Theorem~\ref{benevides} implies an asymptotic version of Schelp's conjecture and provides not only monochromatic paths but also equally long monochromatic cycles. Thus Theorem~\ref{benevides} yields a partial result towards the following question of Li, Nikiforov and Schelp~\cite{LNS}:

\begin{question}[\cite{LNS}]\label{LNSq}
Let $0<c<1$ and $n$ be sufficiently large integer and 
$G$ be  a $2$-edge-colored  graph of order $n$ with $\delta(G)>c n$.
What is the minimum possible  monochromatic circumference of $G$?
\end{question}

{Benevides, \L uczak, Scott, Skokan and White~\cite{BLSSW} made the following definition and White~\cite{W} gave an asymptotic answer to Question~\ref{LNSq}.}

\begin{defn}
{For any positive integer $r$ and $0<c<1$, let $\Phi_r(c)$ be the supremum of
the set of real-valued $\phi$ such that any $r$-edge-colored graph $G$ of sufficiently large order $n$ with minimum degree at least $cn$ has monochromatic circumference at least $\phi n$.}
\end{defn}

\begin{theorem}[\cite{W}]
{For any $c \in (0,1)$, let $m_c$ be the unique integer such that $c \in  (\frac{2m_c+1}{(m_c+1)^2}, \frac{2m_c - 1}{m_c^2}]$. Then}
$   
\Phi_2(c) = 
     \begin{cases}
     \frac{2}{3} \hspace{4.85cm} \text{ for }c \in [\frac{3}{4}, 1), \\
     \min\{\frac{1}{m_c}, \frac{2}{3} c\} \hspace{3cm} \text{ for }c \in (0, \frac{3}{4}) \text{ except }c = \frac{5}{9},
      \\ \frac{10}{27} \hspace{4.7cm} \text{ for }c = \frac{5}{9}.
     \end{cases}
$
\end{theorem}

 Benevides, \L uczak, Scott, Skokan and White~\cite{BLSSW} also conjectured the following.

\begin{conjecture}[\cite{BLSSW}]\label{BLSSWconj}
Let $G$ be a graph of order $n$ with $\delta(G) \ge 3n/4$. Let $n = 3t+r$, where $r \in \{0,1,2\}$. Every $2$-edge-coloring $E(G) = E(R_G) \cup E(B_G)$ of $G$ has monochromatic circumference at least $2t+r$.
\end{conjecture}

\section{Results}

Our main result is the following theorem in the direction outlined by Schelp and in
the spirit of Theorems~\ref{12} and~\ref{benevides}.

\begin{theorem}\label{maintheorem2}
There exists a positive integer $n_0$ with the following property. Let $n = 3t+r>n_0$, where $r \in \{0,1,2\}$. Let $G$ be a graph of order $n $ with 
$\delta(G) \ge (3n-1)/4$. 
 Then for every  $2$-edge-coloring of $G$,  either there are cycles of every length in $\{3, 4, 5, \dots, 2t+r\}$ of the same  color, or 
there are cycles of every even length in $\{4, 6, 8, \dots, 2t+2\}$ of the same color.
\end{theorem}

The following  examples show that our result is best possible.

\textbf{Example 1:} Let $G$ be the complete graph on $3t+r$ vertices, where $t$ is a positive integer and $r \in \{0,1,2\}$. We partition the vertex set of $G$ into $U_1$ and $U_2$ such that $|U_1| = 2t+r$ and $|U_2| = t$. Color all edges inside $U_1$ and $U_2$ blue and all edges in $G[U_1, U_2]$ red. There is a blue cycle of length $2t+r$ but no monochromatic cycle of length larger  than $2t+r$.

\textbf{Example 2:} Let $G$ be a graph on $n = 3t+r$ vertices, where $t$ is a positive integer and $r \in \{0,1,2\}$. We partition $V(G)$ into $U_1 \cup U_2 \cup U_3 \cup U_4 \cup \{x,y\}$, where $\lfloor \frac{n-2}{4} \rfloor \le |U_1| \le |U_2| \le |U_3| \le |U_4| \le \lceil \frac{n-2}{4} \rceil$. We obtain $G$ from $K_{3t+r}$ by deleting all edges between $U_1$ and $U_4$, and all edges between $U_2$ and $U_3$. Color all edges in $G[U_1, U_2]$ and $G[U_3, U_4]$ blue, all edges in $G[U_1, U_3]$ and $G[U_2, U_4]$ red, all edges incident with $x$ red and all edges incident with $y$ apart from the edge $xy$ blue; edges in $G[U_1], G[U_2], G[U_3]$, and $G[U_4]$ may be colored arbitrarily. The minimum degree in $G$ is 
\[
	n-1-\left\lceil \frac{n-2}{4} \right\rceil =\left\lfloor\frac{3n-2}{4}\right\rfloor,
\]
and a longest monochromatic cycle has length at most $2 \lceil \frac{n-2}{4} \rceil +1$, which is strictly less than $2t+r$ for $n \ge 8$.

\textbf{Example 3:} Let $G$ be a complete graph with $n = 3t+1$ vertices, where $t$ is a positive integer. We partition the vertex set of $G$ into $U_1$ and $U_2$ such that $|U_1| = \lfloor \frac n2 \rfloor$ and $|U_2| = \lceil \frac n2 \rceil$. Color all edges inside $U_1$ and $U_2$ blue and all edges in $G[U_1, U_2]$ red. Although $G$ has a red cycle of length $2 \lfloor \frac n2\rfloor$, there is no monochromatic cycle of length exactly $2t+1$, since the red graph is bipartite and the largest component in the blue graph only contains $\lceil \frac{n}{2} \rceil < 2t+1$ vertices. In particular, there is no monochromatic cycle of length exactly $2t+1$ in $G$.

Thus, the conditions of Theorem~\ref{maintheorem2} for $r=1$ can neither guarantee a monochromatic cycle of length $2t+1$ in $G$ nor a monochromatic cycle of length $2t+2$. However,  they imply that $G$ has a monochromatic cycle of at least one of these lengths, and of many other lengths.

Since $2t+2 \ge 2t+r$, Theorem~\ref{maintheorem2} immediately yields a  slightly stronger version of  Conjecture~\ref{BLSSWconj} 
(with restriction $\delta(G) \ge (3n-1)/4$ in place of $\delta(G) \ge 3n/4$) for  every sufficiently large $n$:

\begin{theorem}\label{maintheorem}
	\label{cycle-theorem}
	There exists a positive integer $n_0$ with the following property.
	Let $G$ be a graph of order $n > n_0$ with $\delta(G) \ge (3n-1)/4$. Let $n = 3t+r$, where $r \in \{0,1,2\}$. Then every $2$-edge-coloring of $G$ contains a monochromatic cycle of length at least $2t+r$.
\end{theorem}

Observe that although Conjecture~\ref{BLSSWconj} is stated for all $n$, it is not true for $n\in \{4,5\}$. Indeed, $E(K_4)$ decomposes into two 
Hamiltonian paths, and so the corresponding edge-coloring of $K_4$ does not have monochromatic cycle at all. Also, $E(K_5)$ decomposes into two
{\em bull graphs} (paths of length $4$ with the chord connecting the second and the fourth vertices); hence the corresponding edge-coloring of $K_5$ does not have 
a monochromatic cycle of length at least $4$.

Theorem~\ref{maintheorem} in turn implies Conjecture~\ref{schelp}:

\begin{theorem}\label{schelpT}
Suppose that $n$ is large enough and $G$ is a graph on $3n-1$ vertices with minimum degree at least $(3|V(G)|-1)/4$. 
Then $G$ arrows~$P_{2n}$.
\end{theorem}

\begin{proof}
We have $3n-1 = 3t+r$ for $t = n-1$ and $r=2$. Theorem~\ref{cycle-theorem} yields that $G$ has a monochromatic cycle of length at least $2t+r = 2n$, so in particular $G$ contains a monochromatic $P_{2n}$.
\end{proof}

 Gy\'{a}rf\'{a}s and  S\'{a}rk\"{o}zy~\cite{gyarfas12} suggested that maybe the claim in Conjecture~\ref{schelp} holds for all $n$.
Theorem~\ref{maintheorem2} is also a (small) step toward a resolution of Question~\ref{LNSq}.




Our proof of Theorem~\ref{maintheorem2} uses the Szemer\' edi Regularity Lemma~\cite{Sz}, the idea of connected matchings in regular partitions due to \L uczak~\cite{L1}, a stability theorem of Benevides, \L uczak, Scott, Skokan and White (Lemma~4.1 in~\cite{BLSSW}, see Lemma~\ref{stability-lemma} in Section~4 below),
and several classical theorems on existence of cycles in graphs, including theorems of Berge~\cite{BGG} and Jackson~\cite{jackson81}. 

We first apply the $2$-color version of the Szemer\'edi Regularity Lemma to $G$ to obtain a reduced graph $H$. Then we apply Lemma~4.1 to obtain three cases. In Case~(i) of Lemma~4.1, it is already shown in~\cite{BLSSW} that there is a long monochromatic cycle, and some additional work yields the conclusions of Theorem~\ref{maintheorem2} as well. The remaining two cases describe near-extremal graphs, which we handle separately: we deal with Case~(ii) in Section~\ref{sparse-set-section-2} and with Case~(iii) in Section~\ref{four-part-section-2}.

\section{Tools}

\subsection{The Regularity Lemma}

{For the sake of consistency, we use the same form of $2$-colored version of the Szemer\'edi Regularity Lemma as in~\cite{BLSSW}, which can be deduced from the standard form of Szemer\'edi Regularity Lemma~\cite{Sz}; the definitions and theorems given there are reproduced below.}

\begin{defn}Let $G$ be a graph and $X$ and $Y$ be disjoint subsets of $V(G)$. The \emph{density} of the pair $(X,Y)$ is the value
\[
	d(X,Y) := \frac{e(X,Y)}{|X||Y|}.
\]
Let $\epsilon > 0$ and  $G$ be a graph and $X$ and $Y$ be disjoint subsets of $V(G)$. We call $(X,Y)$ an \emph{$\epsilon$-regular pair for $G$} if, for all $X' \subseteq X$ and $Y' \subseteq Y$ satisfying $|X'| \ge \epsilon |X|$ and $|Y'| \ge \epsilon |Y|$, we have
\[
	|d(X,Y) - d(X', Y')| < \epsilon.
\]
\end{defn}

\begin{theorem}[Theorem 2.4 in \cite{BLSSW}]
\label{regularity-lemma}
For every $\epsilon>0$ and positive integer $k_0$, there is an $M = M(\epsilon, k_0)$ such that if $G = (V,E)$ is an arbitrary $2$-edge-colored graph and $d \in [0,1]$, then there is $k_0 \le k \le M$, a partition $(V_i)_{i=0}^k$ of the vertex set $V$ and a subgraph $G' \subseteq G$ with the following properties:
\begin{enumerate}
\item [(R1)] $|V_0| \le \epsilon|V|$, 
\item [(R2)] all clusters $V_i$, $i \in [k] := \{1,2,\dots,k\}$, are of the same size $m \le \lceil \epsilon|V|\rceil$, 
\item [(R3)] $\deg_{G'}(v) > \deg_G(v) - (2d+\epsilon)|V|$ for all $v \in V$, 
\item [(R4)] $e(G'[V_i]) = 0$ for all $i \in [k]$,
\item [(R5)] for all $1 \le i < j \le k$, the pair $(V_i, V_j)$ is $\epsilon$-regular for $R_{G'}$ with a density either $0$ or greater than $d$ and $\epsilon$-regular for $B_{G'}$ with a density either $0$ or greater than $d$, where $E(G') = E(R_{G'}) \cup E(B_{G'})$ is the inherited $2$-edge-coloring of $G'$.
\end{enumerate}	
\end{theorem}

\begin{defn}
Given a graph $G = (V,E)$ and a partition $(V_i)_{i=0}^k$ of $V$ satisfying conditions (R1)--(R5) above, we define the \emph{$(\epsilon, d)$-reduced $2$-edge-colored graph $H$} on vertex set $\{v_i : 1 \le i \le k\}$ as follows. For $1 \le i < j \le k$, 
\begin{itemize}
\item let $v_iv_j$ be a blue edge of $H$ when $B_{G'}[V_i, V_j]$ has density at least $d$;
\item let $v_iv_j$ be a red edge of $H$ when $R_{G'}[V_i, V_j]$ has density at least~$d$.
\end{itemize}
\end{defn}

Our definition of the reduced graph departs slightly from the definition of \cite{BLSSW}: we let an edge $v_iv_j$ of $H$ have both red and blue colors  when $B_{G'}[V_i, V_j]$ and $R_{G'}[V_i, V_j]$ both are $\epsilon$-regular and have density at least $d$, while in such cases, it is only a blue edge in~\cite{BLSSW}. We can also view this reduced graph as a red graph $H_R$ and a blue graph $H_B$ whose edge sets are not necessarily disjoint; this observation is especially helpful in Section~\ref{lemma-extension}.

\subsection{Extremal results for matchings, paths, and cycles.}
We will use the following extremal results to find cycles of desired length in Section~\ref{sparse-set-section-2} and Section~\ref{four-part-section-2}.

Theorems~\ref{bagga}, \ref{berge}, \ref{bondy}, \ref{chvatal}, and \ref{jackson} will all be used to show the existence of cycles in the reduced graph. The variety is necessary for handling differently structured graphs, and in particular, Theorem~\ref{jackson} is one of only a few such results for unbalanced bipartite graphs. Theorem~\ref{bondys} will be used to find cycles too short to be found using the Szemer\'edi Regularity Lemma.

\begin{theorem}[Bagga and Varma~\cite{BV}]\label{bagga}
Let $G$ be a bipartite balanced graph of order $2n$ such that the sum of the degrees of any  two non-adjacent vertices from different parts is at least $n+1$. Then $G$ is bipancyclic.
\end{theorem}

\begin{theorem}[Berge~\cite{BGG}]\label{berge}
Let $H$ be a $2m$-vertex bipartite graph with vertices $u_1, u_2, \ldots, u_m$ on one side and $v_1, v_2, \ldots, v_m$ on the other, such that $\deg(u_1) \le \ldots \le \deg(u_m)$ and $\deg(v_1) \le \ldots \le \deg(v_m)$. Suppose that for the smallest two indices $i$ and $j$ such that $\deg(u_i) \le i+1$ and $\deg(v_j) \le j+1$, we have $\deg(u_i) + \deg(v_j) \ge m+2$. Then $H$ is Hamiltonian bi-connected: for every $i$ and $j$, there is a Hamiltonian path with endpoints $u_i$ and $v_j$.
\end{theorem}

\begin{theorem}[Bondy~\cite{bondy}]\label{bondy}
Let $G$ be a graph of order $n$ such that for every pair of non-adjacent vertices has their degree sum at least $n$. Then $G$ is either pancyclic or $G$ is the bipartite complete graph $K_{\lceil \frac{n}{2} \rceil, \lfloor \frac{n}{2} \rfloor}$.
\end{theorem}

\begin{theorem}[Bondy and Simonovits~\cite{BS}]\label{bondys}
Let $G$ be a graph on $n$ vertices with $|E(G)| > 100qn^{1+\frac{1}{q}}$. Then $G$ contains cycles of every even  length from $[2q, 2n^{\frac{1}{q}}]$.
\end{theorem}

\begin{theorem}[Chv\'atal~\cite{chvatal72}; see also Corollary~5 in Chapter~10 in \cite{BGG}] 
\label{chvatal}
Let $G$ be a graph of order $n\ge 3$ with degree sequence $d_1 \le d_2 \le \ldots \le d_n$ such that
\[
	d_k \le k < \frac n2 \implies d_{n-k} \ge n-k.
\]
Then $G$ contains a Hamiltonian cycle.
\end{theorem}

\begin{theorem}[Jackson~\cite{jackson81}]\label{jackson}
Let $G$ be a bipartite graph with bipartition $(X,Y)$ in which every vertex of $X$ has degree at least $k$. If $2 \le |X|\le k$ and $|Y| \le 2k-2$, then $G$ contains a cycle of length $2|X|$.
\end{theorem}

In addition, Theorem~\ref{hall} is a standard result on matchings in bipartite graphs.
 
\begin{theorem}[Hall]
\label{hall}
Let $H$ be a bipartite graph with bipartition $(X,Y)$ with $|X| \le |Y|$. If $|N(S)| \ge |S|$ for every $S \subseteq X$, then $H$ has a matching saturating $X$.
\end{theorem}

\section{Main part of the proof of Theorem~\ref{maintheorem2}}

\begin{proof}[Proof of Theorem~\ref{maintheorem2}:]
We begin with a stability result from~\cite{BLSSW}:

\begin{lemma}[Lemma 4.1 in \cite{BLSSW}]\label{stability-lemma}
Let $0 < \delta < 1/36$ and let $G$ be a graph of sufficiently large order $k$ with $\delta(G) \ge (3/4 - \delta)k$. Suppose that we are given a $2$-edge-coloring $E(G) = E(R) \cup E(B)$. Then one of the following holds.

\begin{enumerate}
\item[(i)]There is a component of $R$ or $B$ that contains a matching on at least $(2/3+\delta)k$ vertices.

\item[(ii)]There is a set $S$ of order at least $(2/3  - \delta/2)k$ such that either $\Delta(R[S]) \le 10\delta k$ or $\Delta(B[S]) \le 10\delta k$.

\item[(iii)]There is a partition $V(G) = U_1 \cup U_2 \cup U_3 \cup U_4$ with $\min_i\{|U_i|\} \ge (1/4 - 3\delta)k$ such that there are no red edges from $U_1 \cup U_2$ to $U_3 \cup U_4$ and no blue edges from $U_1 \cup U_3$ to $U_2 \cup U_4$.
\end{enumerate}	
\end{lemma}

Because our definition of the reduced graph is slightly different from the one in~\cite{BLSSW}, we will need to apply Lemma~\ref{stability-lemma} to a slightly more general class of graphs: $2$-edge-colored graphs in which an edge can potentially be colored both red and blue. It could be checked that the proof of Lemma~\ref{stability-lemma} in~\cite{BLSSW} continues to work:  it never uses the existence of an edge $vw \in E(R)$ to conclude that $vw \notin E(B)$. But then the reader would need to read a 4-page proof in~\cite{BLSSW}. To avoid this, in Section~\ref{lemma-extension}, we present a different argument to extend Lemma~\ref{stability-lemma}. It results in a smaller maximum value of $\delta$, but this will not matter, since below, we will take $\delta < \frac1{1000}$.

{Choose $0 < 10^{20} \cdot \epsilon < 10^{10} \cdot d < \delta < \frac1{1000}$} and a sufficiently large $n_0$ as in~\cite{BLSSW}. We let $G$ be a graph satisfying the hypotheses of Theorem~\ref{cycle-theorem}, and apply Theorem~\ref{regularity-lemma} to obtain an $\epsilon$-regular partition of $V(G)$, {a graph $G'$ satisfying conditions $R1-R5$ in Theorem~\ref{regularity-lemma},} and a reduced graph $H$. Since $d$ and $\epsilon$ are much smaller than $\delta$, the minimum degree in $H$ is at least $(\frac34 - \delta)k$, and each $v \in V(G)$ is incident to at most $\delta n$ edges not present in the subgraph $G'$ provided by Theorem~\ref{regularity-lemma}.

When we apply Lemma~\ref{stability-lemma} to $H$, there are three possibilities.

If Case~(i) of Lemma~\ref{stability-lemma} holds, then it is already shown in Lemma~2.5 of~\cite{BLSSW} that $G$ contains a monochromatic cycle, say red, of length $\ell$ for all even $\ell$ such that $4k \le \ell \le (2/3 + \delta/2)n$; in particular, of every even length from $[4k, 2t+2]$. 

Since a red matching edge in $H$ corresponds to an $\epsilon$-regular $d$-dense pair $(V_i, V_j)$ in $G$, where $(1-\epsilon)\frac{n}{k} \le |V_i| = |V_j| \le \frac{n}{k}$, there are at least $d |V_i| |V_j| \ge d(1-\epsilon)^2\frac{n^2}{k^2}> 200(\frac{2n}{k})^{\frac{3}{2}}$ edges in $R[V_i, V_j]$. By Theorem~\ref{bondys}, we have a red cycle of every even  length in $[4, 2\sqrt{\frac{2n(1-\epsilon)}{k}}]$. Since $4k \ll 2\sqrt{\frac{2n(1-\epsilon)}{k}}$, there is a red cycle of every even length in $[4,2t+2]$. {Note that the proof of Case~(i) can be viewed as an application of the blow-up lemma~\cite{KSS}.}

Suppose that Case~(ii) of Lemma~\ref{stability-lemma} holds. Let $L \subseteq V(G)$ be the union of all clusters $V_i$ such that the vertex $v_i$ of the reduced graph was an element of the set $S$ found in Case~(ii). We have $|L| \ge (2/3 - \delta/2)k|V_i|$ (where $i \in [k]$ is arbitrary), and $|V_i| = (n - |V_0|)/k \ge (1-\epsilon)n/k$, hence $|L| \ge (2/3 - \delta/2 - \epsilon)n \ge (2/3 - \delta)n$. 

Without loss of generality, it is the red edges that are sparse inside $S$, in which case $\Delta(R_H[S]) \le 10 \delta k$. For a cluster $V_i \subseteq L$, there are at most $10\delta k$ parts $V_j$, $1 \le j \le k$, such that $V_j \subseteq L$ and the density of the $\epsilon$-regular pair $(V_i, V_j)$ is greater than $d$. They contribute at most $10 \delta k \cdot \frac nk = 10\delta n$ to the red degree of a vertex in $V_i$. For all other parts $V_j \subseteq L$, the pair $(V_i, V_j)$ is $\epsilon$-regular with density $0$ in $R_{G'}$, which means that there are no red edges between $V_i$ and $V_j$ in $G'$; neither are there edges within $V_i$. Finally, each $v \in L$ has at most $\delta n$ edges in $G$ which are not in $G'$. Therefore $L \subseteq V(G)$ satisfies $\Delta(R_G[L]) \le 11\delta n$. 

We complete this case of the proof of Theorem~\ref{maintheorem2} with the following lemma, whose proof is given in Section~\ref{sparse-set-section-2}.

\begin{lemma}\label{sparse-set-2}
Let $0 < \delta < \frac1{1000}$, and let $G$ be a graph of order $n$ with $\delta(G) \ge (3n-1)/4$ with a $2$-edge-coloring $E(G) = E(R) \cup E(B)$.  Let $n = 3t+r$, where $r \in \{0,1,2\}$. Suppose that there is a set $L \subseteq V(G)$ of order at least $(2/3 - \delta)n$ such that $\Delta(R[L]) \le 11\delta n$. Then either one of $R$ and $B$  contains cycles of every integer length in $[3,2t+r]$ or  one of $R$ and $B$ contains cycles of every even  length in $[4,2t+2]$.
\end{lemma}

Finally, suppose that Case~(iii) of Lemma~\ref{stability-lemma} holds. In this case, for $j=1,2,3,4$, let $\mc U_j$ be the union of all clusters $V_i$ such that the vertex $v_i$ of the reduced graph was an element of the set $U_j$ found in Case~(iii).

For each $j$, we have
\[
	|\mc U_j| \ge (1/4 - 3\delta)k \cdot \frac{n - |V_0|}{k} \ge (1/4 - 3\delta)(1-\epsilon)n \ge (1/4 - 4\delta)n.
\]
The graph $R_{G'}[\mc U_1 \cup \mc U_2, \mc U_3 \cup \mc U_4]$ is empty: if $v \in V_i \subseteq \mc U_1 \cup \mc U_2$ and $w \in V_j \subseteq \mc U_3 \cup \mc U_4$, then there cannot be a red edge between $v_i$ and $v_j$ in $H$, which means that the pair $(V_i, V_j)$ is $\epsilon$-regular with density $0$ in $R_{G'}$: there are no red edges in $G'$ between $V_i$ and $V_j$. In particular, $vw$ cannot be a red edge in $G'$.  Every vertex  in $G$ is incident to at most $\delta n$ edges not  in $G'$. Therefore $R_G[\mc U_1 \cup \mc U_2, \mc U_3 \cup \mc U_4]$ has maximum degree at most $\delta n$. Similarly, $B_G[\mc U_1 \cup \mc U_3, \mc U_2 \cup \mc U_4]$ has maximum degree at most $\delta n$.

The set $V_0$ in $G$ is not a part of any $\mc U_j$, but  $|V_0| \le \epsilon|V| \le \delta|V|$ by (R1).

We complete this case by the following lemma, whose proof is given in Section~\ref{four-part-section-2}.

\begin{lemma}\label{four-parts-2}
Let $0 < \delta < \frac1{1000}$, and let $G$ be a graph of order $n = 3t + r$, where $r \in \{0,2\}$, with $\delta(G) \ge (3n-1)/4$ and a $2$-edge-coloring $E(G) = E(R) \cup E(B)$. Suppose that there is a partition $V(G) = \mc U_1 \cup \mc U_2 \cup \mc U_3 \cup \mc U_4 \cup V_0$ such that 
\begin{itemize}
\item $(1/4 - 4\delta)n \le |\mc U_j|$ for each $j$, $|V_0| \le \delta n$, and
\item $R[\mc U_1 \cup \mc U_2, \mc U_3 \cup \mc U_4]$ and $B[\mc U_1 \cup \mc U_3, \mc U_2 \cup \mc U_4]$ have maximum degree at most $\delta n$.
\end{itemize}
Then one of $R$ and $B$ contains
 cycles of every even  length in $[4,2t+2]$.
\end{lemma}
This exhausts all cases of Lemma~\ref{stability-lemma}, completing the proof of Theorem~\ref{maintheorem2}.
\end{proof}

\section{Proof of Lemma~\ref{sparse-set-2}}\label{sparse-set-section-2}
In this section, we assume that there is a set $L \subseteq V(G)$ of order at least $\ell = (2/3 - \delta)n$ such that $\Delta(R[L]) \le 11\delta n$. We write $n=3t+r$, where $r \in \{0, 1, 2\}$.

We consider two cases; in the first case, we find blue cycles and in the other case, red cycles. {Let $L'$ be the set of vertices in $V(G)-L$ with at least $ \delta n +2$ blue edges to $L$.}

\textbf{Case 1:} {$|L \cup L'| \ge 2t+r$.} 


We begin by finding blue cycles of every length from $\{3, 4, \dots, |L|\}$. Since $\Delta(R[L]) \le 11\delta n$, the minimum degree in $B[L]$ is at least $|L|- 1 - \frac{n-3}{4} - 11 \delta n \ge 0.6 |L|$. For any two vertices $u,v \in L$, their degrees in $B[L]$ sum to more than $|L|$. 
Hence  $B[L]$ is pancyclic by Theorem~\ref{bondy} {and we obtain blue cycles of every length in $\{3, \ldots, |L|\}$. }

{If $|L| \ge 2t + r$, then we are done. Otherwise, $|L| < 2t+r$ holds, and we are still required to find blue cycles of the missing lengths $\{|L|+1, \ldots, 2t+r\}$. Let $k \in \{|L|+1, \ldots, 2t+r\}$; we will show that a blue cycle of length $k$ exists. 

Let $Y$ be obtained from $L$ by the addition of $k-|L|$ {vertices in $L'$}; let $d_1 \le d_2 \le \ldots \le d_k$ be the degree sequence of $B[Y]$. A Hamiltonian cycle in $B[Y]$ will give us a blue cycle of length~$k$. 

We verify that $d_i \ge i+1$ for all $i \le k/2$. If the vertex of degree $d_i$ was originally in $L$, then 
\[
	d_i \ge |L| - 11\delta n - (n-3)/4 \ge (5/12 - 11\delta)n \ge 0.405n,
\]
while $k/2 +1 \le 0.334n$, so $d_i \ge k/2+1 \ge i+1$. Therefore, if $d_i \le k/2$, then we are looking at a vertex of $v_i \in Y - L$ and $i  \le k - |L| \le k - \ell \le \delta n +r/3 \le \delta n + 1$. But then, $d_i \ge  \delta n + 2 \ge i+1$ by our choice of vertices to add to $Y$. By Theorem~\ref{chvatal}, $B[Y]$ contains a Hamiltonian cycle, which is a blue cycle of length exactly $k$, as desired. Thus, we find blue cycles of every length from $\{3, 4, \dots, 2t+r\}$.

\textbf{Case 2:}  {$|L \cup L'| < 2t + r$. This leaves at least $t+1$ vertices in $V(G)-L$ that have at most $ \delta n + 1$ blue edges to $L$.} 

Let $2m \in \{4, 6, \dots, 2t+2\}$. Let $X \subseteq V(G) - L $ consist of $m$ vertices, each with fewer than $\delta n + 2$ blue edges into $L$. { We claim that in the bipartite graph $R[X, L]$, every vertex $x \in X$ has degree at least $(5/12-3\delta)n$. To see this, there are at least $\ell$ vertices of $L$, $x$ is not adjacent in $G$ to at most $n-1-(3n-1)/4 < n/4$ of them, and $x$ has blue edges to fewer than $\delta n + 2$  vertices. Thus, each $x \in X$ has degree at least $\ell-n/4-\delta n - 2 \ge (2/3-\delta)n - n/4-\delta n - 2 \ge (5/12-3\delta)n$.}

Our goal is to apply Theorem~\ref{jackson} with $k = (5/12-3\delta)n$ to the graph $R[X,L]$. We have already checked that every vertex of $X$ has degree at least $k$. We verify the other two conditions:
\[
	|X| \le m \le t+1 \le n/3 + 1 \le (5/12 - 3\delta)n = k,
\]
and
\[
	|L| \le n - (t+1) \le 2n/3 \le (5/6 - 6\delta)n - 2 = 2k - 2.
\]
Therefore $R[X,L]$ contains a cycle of length $2m$, and as $m$ varies, we obtain a red cycle of every even length from $\{4, 6, \dots, 2t+2\}$.

\section{Proof of Lemma~\ref{four-parts-2}}\label{four-part-section-2}
We have a partition of $V(G)$ into $\mc U_1 \cup \mc U_2 \cup \mc U_3 \cup \mc U_4 \cup V_0$ such that 
\begin{equation}\label{61}
\mbox{ $(1/4 - 4\delta)n \le |\mc U_j|$ for each $j$, $|V_0| \le \delta n$, and}
\end{equation}
\begin{equation}\label{62}
\mbox{each of
 $R[\mc U_1 \cup \mc U_2, \mc U_3 \cup \mc U_4]$ and $B[\mc U_1 \cup \mc U_3, \mc U_2 \cup \mc U_4]$ has maximum degree at most $\delta n$.}
\end{equation}

\begin{defn}
Let $G$ be a bipartite graph with parts $X$ and $Y$. The {\em deficiency}, $\overline{d}(v)$ of a vertex $v$ is $|Y|-\deg(v)$ when $v \in X$ and $|X|-\deg(v)$ when $v\in Y$.
\end{defn}

\begin{lemma}
\label{non-degree}
In each of the graphs $R[\mc U_1, \mc U_2]$, $R[\mc U_3, \mc U_4]$, $B[\mc U_1, \mc U_3]$, and $B[\mc U_2, \mc U_4]$, every vertex has deficiency at most $7\delta n$.
\end{lemma}
\begin{proof}
Without loss of generality, consider the graph $R[\mc U_1, \mc U_2]$ and let $v \in \mc U_1$. An edge from $v$ to $\mc U_4$ would be in either $R[\mc U_1 \cup \mc U_2, \mc U_3 \cup \mc U_4]$ or $B[\mc U_1 \cup \mc U_3, \mc U_2 \cup \mc U_4]$, each of which by~\eqref{62} has maximum degree at most $\delta n$; so there can be at most $2\delta n$ such edges.  Since $|\mc U_4| \ge (1/4 - 4\delta)n$, there are at least $(1/4-6\delta)n$ vertices in $\mc U_4$ not adjacent to $v$. 
Since $\delta(G)\geq (3n-1)/4$,  there are at most $(n-3)/4<n/4$ vertices not adjacent to $v$; therefore $v$ has deficiency at most $6\delta n$ in $G[\mc U_1, \mc U_2]$. Finally, each blue edge of $v$ in $G[\mc U_1,\mc  U_2]$ is also in $B[\mc U_1\cup \mc U_3, \mc U_2 \cup \mc U_4]$, so by~\eqref{62} there are at most $\delta n$ such edges, and the deficiency of $v$ in $R[\mc U_1, \mc U_2]$ is at most $7\delta n$.
\end{proof}


We first find monochromatic cycles of every even length from $[4, (\frac12 - 8\delta)n]$, in both, $R$ and $B$. For red cycles, consider $R[\mc U_1, \mc U_2]$.  We pick a set $X \subseteq \mc U_1$ and a set $Y \subseteq \mc U_2$ such that $|X| = |Y| = (\frac14 - 4\delta)n$. By Lemma~\ref{non-degree}, each vertex in $X$ has deficiency at most $7 \delta n$ in $R[X, Y]$ to $Y$ and each vertex in $Y$ has deficiency at most $7 \delta n$ in $R[X, Y]$ to $X$. Since the degrees of any pair of non-adjacent vertices in $R[X, Y]$ sum to at least $(\frac14 - 4\delta)n + 1$, $R[X, Y]$ is bipancyclic by Theorem~\ref{bagga}. We obtain blue cycles in $B[\mc U_1, \mc U_3]$ by the same argument.

In the remainder of this section, we show that either $R$ or $B$ contains cycles of every even length from $[(\frac18 - \delta)n, 2t+2]$. 
First,  we need to prove some preliminary lemmas.


\begin{lemma}\label{bi-connected2}
Let $H$ be a bipartite graph with parts $A_1$ and  $A_2$, where $|A_1|, |A_2| \ge (\frac14 - 5\delta)n$, and assume every vertex of $H$ has deficiency at most $10\delta n$. Then
\begin{enumerate}
\item \label{odd-path} For each odd  $\ell\in [(\frac14 - 4\delta)n - 5, t+5]$ and any vertices $x_1 \in A_1$, $x_2 \in A_2$, there is an $(x_1, x_2)$-path in $H$ of length exactly $\ell$.

\item \label{even-path} For each even  $\ell\in [(\frac14 - 4\delta)n - 5, t + 5]$ and any vertices $x_1, x_1' \in A_1$, there is an $(x_1, x_1')$-path in $H$ of length exactly $\ell$.
\end{enumerate}
\end{lemma}
\begin{proof}
We prove a stronger result: that the conclusion of the lemma holds for all $\ell$ satisfying $80\delta n + 3 \le \ell \le (\frac12 - 10\delta)n - 1$.

To prove Statement~\ref{odd-path}, we pick a set of vertices $X_1 \subseteq A_1$ such that $|X_1| = \frac12 (\ell+1)$ and $x_1 \in X_1$, and a set of vertices $X_2 \subseteq A_2$ such that $|X_2| = \frac12 (\ell+1)$ and $x_2 \in A_2$, noting that $|A_i| \ge (\frac14 - 5\delta)n \ge \frac12(\ell+1)$ for $i=1,2$. Since every vertex in $H$ has deficiency at most $10\delta n$, the same is true for $H' := H[X_1, X_2]$, and therefore every vertex of $H'$ has degree at least  $\frac12 (\ell+1) - 10\delta n$.

In particular, for any two vertices $u \in X_1$, $v \in X_2$, 
\[
	\deg_{H'}(u) + \deg_{H'}(v) \ge (\ell+1) - 20\delta n \ge \frac12 (\ell+1) + 2,
\]
and therefore $H'$ is Hamiltonian bi-connected by Theorem~\ref{berge}. In particular, $H'$ contains a Hamiltonian $(x_1,x_2)$-path, which has length $\ell$.

To prove \ref{even-path}, we first pick any $x_2 \in A_2$ adjacent to $x_1'$, then proceed as above with subsets $X_i \subseteq A_i$ of size $\frac12\ell$, making sure that $x_1' \notin A_1$. The same argument finds an $(x_1,x_2)$-path of length $\ell-1$, which extends to an $(x_1,x'_1)$-path of length $\ell$ with the addition of the  edge $x_2x_1'$.
\end{proof}

\begin{lemma}\label{2matching}
For every even length $2\ell\in [(\frac12 - 8\delta)n, 2t+2]$, we can find a red cycle (blue in Case 4) of length exactly $2\ell$ in $G$ in the following cases.
\begin{enumerate}
\item Both $R[\mc U_1, \mc U_3]$ and $R[\mc U_2, \mc U_4]$ contain at least one edge: two red edges $x_1y_1$ and $x_2y_2$ with $x_1 \in \mc U_1$, $y_1 \in \mc U_3$, $x_2 \in \mc U_2$ and $y_2 \in \mc U_4$.

\item We find an edge in each of $R[\mc U_1, \mc U_4]$ and $R[\mc U_2, \mc U_3]$, or we find a matching of size $2$ in any of $R[\mc U_i, \mc U_j]$ where $i \in \{1,2\}$ and $j \in \{3,4\}$.

\item As in Cases 1 and 2, but with the edges $x_1y_1, x_2y_2$ or the edges of the matching replaced by vertex-disjoint paths of length $2$ with no interior vertices in $\mc U_1 \cup \mc U_2 \cup \mc U_3 \cup \mc U_4$.

\item As in Cases 1, 2, and 3, but with the corresponding blue structures between $\mc U_1 \cup \mc U_3$ and $\mc U_2 \cup \mc U_4$. 
\end{enumerate} 
\end{lemma}
\begin{proof}
We prove only Case $1$, since  the proofs in Cases $2$, $3$, and $4$ are similar. If $\ell$ is even, then by Lemma~\ref{bi-connected2}, we can find a red $(x_1,x_2)$-path $P_1$ of length $\ell-1$ in $R[\mc U_1, \mc U_2]$ and a red $(y_1,y_2)$-path $P_2$ of length $ \ell-1$ in $R[\mc U_3, \mc U_4]$. If $\ell$ is odd, we find paths of length $\ell$ and $\ell-2$ instead. We then connect $P_1$ and $P_2$ by adding the edges $x_1y_1$ and $x_2y_2$ to obtain a red cycle of length exactly~$2\ell$.
\end{proof}

Suppose $R[\mc U_1 \cup \mc U_2, \mc U_3 \cup \mc U_4]$ contains a matching $M$ of size 3. We claim that in this case 
one of the cases in Lemma~\ref{2matching}  occurs. Suppose otherwise.
Since Case 2 of the lemma does not hold, all edges of $M$ are in distinct $R[\mc U_i, \mc U_j]$ where $i \in \{1,2\}$ and $j \in \{3,4\}$.
By symmetry, we may assume an edge in $M$ is in $R[\mc U_1, \mc U_3]$. Then either we obtain Case 1, or else the other two are not in $R[\mc U_2, \mc U_4]$,
and we have Case 2 of the lemma. Thus, if $R[\mc U_1 \cup \mc U_2, \mc U_3 \cup \mc U_4]$ has a matching of size $3$, then we have a red cycle of every even length from $[(\frac12 - 8\delta)n, 2t+2]$.

Thus, it is enough to consider the situation when neither $R[\mc U_1 \cup \mc U_2, \mc U_3 \cup \mc U_4]$ nor (by symmetry) 
$B[\mc U_1 \cup \mc U_3, \mc U_2 \cup \mc U_4]$ has a matching of size $3$.
In this case, each of them has a vertex cover of size at most $2$. 
 Move the vertices in these vertex covers to $V_0$. Increasing $|V_0|$ by at most $4$, we ensure that both $R[\mc U_1 \cup \mc U_3, \mc U_2 \cup \mc U_4]$ and $B[\mc U_1 \cup \mc U_2, \mc U_3 \cup \mc U_4]$ are empty.

Next, let $X_R = X_B = \emptyset$. We will process the vertices of $V_0$ one at a time, adding each of them to one of $\mc U_1, \mc U_2, \mc U_3, \mc U_4, X_R, X_B$.

Pick a vertex  $v \in V_0$. 
\begin{enumerate}
\item If $v$ has at least three red edges to each of $\mc U_1 \cup \mc U_2$ and $\mc U_3 \cup \mc U_4$, we move $v$ from $V_0$ to $X_R$.

\item If $v$ has at least three blue edges to each of $\mc U_1 \cup \mc U_3$ and $\mc U_2 \cup \mc U_4$, we move $v$ from $V_0$ to $X_B$.

\item If $v$ has at most two red edges to $\mc U_1 \cup \mc U_2$ and at most two blue edges to $\mc U_1 \cup \mc U_3$, we move $v$ from $V_0$ to $\mc U_{4}$.

\item If $v$ has at most two red edges to $\mc U_1 \cup \mc U_2$ and at most two blue edges to $\mc U_2 \cup \mc U_4$, we move $v$ from $V_0$ to $\mc U_{3}$.

\item If $v$ has at most two red edges to $\mc U_3 \cup \mc U_4$ and at most two blue edges to $\mc U_1 \cup \mc U_3$, we move $v$ from $V_0$ to $\mc U_{2}$.

\item If $v$ has at most two red edges to $\mc U_3 \cup \mc U_4$ and at most two blue edges to $\mc U_2 \cup \mc U_4$, we move $v$ from $V_0$ to $\mc U_{1}$.
\end{enumerate}
Note that these conditions check edges from $v$ to the sets $\mc U_1, \mc U_2, \mc U_3, \mc U_4$ at that moment, including vertices of $V_0$ that have already been processed and added to these sets.

At each step, $R[\mc U_1 \cup \mc U_2, \mc U_3 \cup \mc U_4]$ and $B[\mc U_1 \cup \mc U_3, \mc U_2 \cup \mc U_4]$, which initially start out empty, gain at most two edges. Therefore once $V_0$ is processed, each of these graphs has at most $2(\delta n+4)$ edges. 

\begin{claim}\label{post-processing}
After $V_0$ is processed, in each of $R[\mc U_1, \mc U_2], R[\mc U_3, \mc U_4]$, $B[\mc U_1, \mc U_3]$, $B[\mc U_2, \mc U_4]$, each vertex has deficiency at most $8\delta n + 4$.
\end{claim}
\begin{proof}
By symmetry, it suffices to consider vertices in $\mc U_4$ and their deficiency in $R[\mc U_3, \mc U_4]$. Let $v$ be such a vertex. If $v$ was already in $\mc U_4$ before $V_0$ was processed, its deficiency was originally at most $7\delta n$ by Lemma~\ref{non-degree}, and increased by at most $\delta n + 4$: the number of vertices in $V_0$. The total deficiency $8\delta n + 4$ satisfies the claim.

Consider the alternative: vertex $v$ was moved to $\mc U_4$ from $V_0$ by Case 3 of the process. Let us consider the states of $\mc U_1, \mc U_2, \mc U_3, \mc U_4, V_0$ immediately after this move.

At that point, $v$ has at most $4$ edges to $\mc U_1$: at most two red edges and at most two blue edges. We have $|\mc U_1| \ge (\frac14 - 4\delta)n-4$, giving us at least $(\frac14 - 4\delta)n - 8$ vertices $v$ is not adjacent to. However, there are at most $n-1 -\delta(G) < \frac14 n$ such vertices total, so $v$ is adjacent to all but $4\delta n + 8$ vertices outside $\mc U_1$. In particular, $v$ has deficiency at most $4\delta n + 8$ in $G[\mc U_3, \mc U_4]$. Since $v$ has at most two blue edges to $\mc U_1 \cup \mc U_3$, $v$ has deficiency at most $4\delta n +10$ in $R[\mc U_3, \mc U_4]$ at this point in the process.

After the remaining part of $V_0$ is processed, the deficiency of $v$ in $R[\mc U_3, \mc U_4]$ can increase by at most $|V_0| \le \delta n + 4$. At the end of the process, its deficiency is still at most $5 \delta n + 14 \le 8\delta n + 4$, as it was claimed.
\end{proof}

By Claim~\ref{post-processing}, Lemma~\ref{bi-connected2} can be applied to the new $\mc U_1, \mc U_2, \mc U_3, \mc U_4$.

Except for at most $4(\delta n+4)$ vertices incident to an edge in either $R[\mc U_1 \cup \mc U_2, \mc U_3 \cup \mc U_4]$ or $B[\mc U_1 \cup \mc U_3, \mc U_2 \cup \mc U_4]$, every vertex $v \in \mc U_{5-j}$, where $j \in [4]$, has no neighbors in $\mc U_j$, so $\deg(v) \le n - 1 - |\mc U_j|$. Recall that we have $\deg(v) \ge (3n-1)/4$ for every $v \in V(G)$. Therefore, $|\mc U_j| \le (n-3)/4$ for every $j$. We have
\[
	|\mc U_1| + |\mc U_2| + |\mc U_3| + |\mc U_4| \le n - 3,
\]
leaving $|X_R| + |X_B| \ge 3$. Without loss of generality, we assume $|X_R| \ge |X_B|$; in particular, $|X_R| \ge 2$.

Call a vertex {\em type $(i,j)$} with $i \in \{1,2\}$ and $j \in \{3,4\}$ if $x$ has two or more red edges to each of $\mc U_i$ and $\mc U_j$. A vertex can be given more than one type, but each vertex in $X_R$ has three red edges to each of $\mc U_1 \cup \mc U_2$ and $\mc U_3 \cup \mc U_4$, and therefore each vertex in $X_R$ is given at least one type.

If there are two vertices in $X_R$ with the same type $(i,j)$ then we can use them to form two red vertex-disjoint paths of length $2$ from $\mc U_i$ to $\mc U_j$. By  Lemma~\ref{2matching}, we can find a red cycle of every even length from $[(\frac12 - 8\delta)n, 2t+2]$, in which case we are done. The same happens if there are two vertices $x, x' \in X_R$ with types $(i,j)$ and $(i',j')$ respectively, where $i \ne i'$ and $j \ne j'$.

The outcome in the previous paragraph can only be avoided if $|X_R| = 2$. In this case, the two vertices in $X_R$ must each have only one type, and the two types agree in only one index. Without loss of generality, the two vertices are $x$ and $x'$ with types $(1,3)$ and $(1,4)$ respectively.

\begin{claim}\label{blueedge}
In this case, either for each $j \in \{2,3,4\}$ every edge with both endpoints in $\mc U_j$ is blue, or there is a red cycle of every even length from $[(\frac12 - 8\delta)n, 2t+2]$.
\end{claim}
\begin{proof}
Suppose that for some $j \in \{2,3,4\}$ there is a red edge $uv$ with both ends in  $\mc U_j$. Consider first the case  $u,v \in \mc U_2$. Let $xa$ and $xb$ be red edges from $x$ to $\mc U_1$ and $\mc U_3$; let $x'a'$ and $x'b'$ be red edges from $x'$ to $\mc U_1$ and $\mc U_4$, with $a' \ne a$.

We could use Lemma~\ref{2matching} to find a red $(a,a')$-path in $R[\mc U_1, \mc U_2]$ and a red $(b,b')$-path in $R[\mc U_3, \mc U_4]$; however, they would join together to a cycle of odd length. To obtain a cycle of even length, we need to use the red edge $uv$.

More precisely, let $2\ell$ be an even length in $[(\frac12 - 8\delta)n, 2t+2]$. By Lemma~\ref{2matching}, there is a red $(b,b')$-path $P_1$ of length 
$2\lceil \ell/2\rceil-1$. 
 To extend $P_1$ to a red cycle of length $2\ell$, we will find a red $(a,a')$-path of length $2\lfloor \ell/2\rfloor - 3$ in $R[\mc U_1, \mc U_2] \cup \{uv\}$.

Let $c$ be a red neighbor of $v$ in $\mc U_1$ {with $c \neq a,a'$.} 
By Claim~\ref{post-processing}, $a$ and $c$ have a common neighbor $d$ in $\mc U_2$. Excluding vertices $\{a,c,d,v\}$ from $R[\mc U_1, \mc U_2]$, we still have a graph to which Lemma~\ref{2matching} applies, and we can find an $(a', u)$-path $P_2$ in that graph of length $2\lfloor \ell/2\rfloor- 7$.
Now we obtain a cycle of length $2\ell$ as the concatenation $P_1, b'x', x'a', P_2, uv, vc, cd, da, ax, xb$.

A similar argument can be applied if the red edge $uv$ is in $\mc U_3$ or $\mc U_4$, except that we find a red $(b,b')$-path in $R[\mc U_3, \mc U_4] \cup \{uv\}]$ using edge $uv$ instead. It is possible that $u$ or $v$ may coincide with $b$ or $b'$, in which case finding the path is even easier.
\end{proof}

From now on, we assume that the first condition of Claim~\ref{blueedge} holds: $R[\mc U_2], R[\mc U_3], R[\mc U_4]$ are empty.

Suppose that one of the vertices in $X_R$, either $x$ or $x'$, could also have been placed in $X_B$ instead, and if we had done so, we would have $|X_B| \ge 2$. If {the argument in Claim~\ref{blueedge}} were repeated for the blue graph $B$, it would be impossible that three out of $\mc U_1, \mc U_2, \mc U_3, \mc U_4$ also contain no blue edges. {To see this, since we know three of them contain no red edges, if three out of $\mc U_1, \mc U_2, \mc U_3, \mc U_4$ also contain no blue edges, then there is an $i \in [4]$ such that $\mc U_i$ contains no edge in $G$. For a vertex $v_i \in U_i$, there can be at most $\delta n + 4$ edges from $v$ to $U_{5-i}$ and thus it is not adjacent to at least $|\mc U_{5-i}| - (\delta n + 4) \ge (\frac{1}{4} - 4\delta)n - 4 - (\delta n + 4) \ge \frac{1}{4}n - 5\delta n - 8$ vertices in $\mc U_{5-i}$. This contradicts $\deg(v) \ge \frac{3n-1}{4}$, since $|\mc U_i| + \frac{1}{4}n - 5 \delta n - 8 > n-1 - \frac{3n-1}{4}$.}

Therefore our  case is that $|X_B| = 1$ and  none of the vertices in $X_R$ could belong in $X_B$. In particular, $x'$, which has type $(1,4)$, could not belong to $X_B$: it either has at most two blue edges to $\mc U_1 \cup \mc U_3$, or at most two blue edges to $\mc U_2 \cup \mc U_4$. Since $x'$ has type $(1,4)$, and is not of type $(1,3)$ and not of type $(2,4)$, $x'$ has at most one red edge to $\mc U_2$ and $\mc U_3$ respectively. Therefore, for some $i \in \{2,3\}$, $x'$ has at most three edges to $\mc U_i$ (at most two blue edges to $\mc U_i$ and at most one red edge to $\mc U_i$).

We now show that this is impossible, ruling out this final case and completing the proof. 

We have $|X_R| = 2$ and $|X_B| = 1$, so $|\mc U_1| + |\mc U_2| + |\mc U_3| + |\mc U_4| = n-3$, which can only happen if $|\mc U_j| = (n-3)/4$ for all $j$.  Except for at most $4(\delta n + 4)$ vertices incident to an edge in either $R[\mc U_1 \cup \mc U_2, \mc U_3 \cup \mc U_4]$  or $B[\mc U_1 \cup \mc U_3, \mc U_2 \cup \mc U_4]$, every vertex in $v \in \mc U_{j}$ has no neighbors in $\mc U_{5-j}$, so it is already missing $(n-3)/4$ edges, and can reach degree $(3n-1)/4$ only if it is adjacent to every vertex in $X_R \cup X_B$. In particular, almost all vertices in both $\mc U_2$ and $\mc U_3$ must be adjacent to $x'$, contradicting the assumption that $x'$ has at most three edges to one of these parts.

\section{Extension of Lemma~\ref{stability-lemma}}\label{lemma-extension}

In this section, we show that Lemma~\ref{stability-lemma} still holds for $2$-edge-colored graphs $G$ if we allow an edge to be both red and blue simultaneously. 

Let $0 < \delta < \frac1{1000}$ and let $G$ be a graph of sufficiently large order $k$ with $\delta(G) \ge (3/4-\delta)k$. Suppose that we are given a $2$-edge-coloring $E(G) = E(R) \cup E(B)$ where $E(R)$ and $E(B)$ are not necessarily disjoint.

For any $2$-edge-coloring $E(G) = E(R') \cup E(B')$ with $E(R') \cap E(B') = \emptyset$, obtained by assigning edges of $E(R) \cap E(B)$ to just one or the other color, we know that Lemma~\ref{stability-lemma} holds. 

If Case~(i) of Lemma~\ref{stability-lemma} holds for any coloring $(R', B')$, then it also holds for the coloring $(R,B)$, since $R'$ and $B'$ are subgraphs of $R$ and $B$, and we are done.

If Case~(iii) of Lemma~\ref{stability-lemma} holds for a coloring $(R', B')$ but does not hold for the coloring $(R,B)$, let $V(G) = U_1 \cup U_2 \cup U_3 \cup U_4$ be the partition we obtain for the coloring $(R', B')$. There are no edge in $G$ between $U_1$ and $U_4$, or between $U_2$ and $U_3$, because there are neither edges in $R'$ nor in $B'$ between those pairs. Therefore, each vertex of $G$ has at least $(1/4-3\delta)k$ missing edges coming from $G[U_1, U_4]$ or $G[U_2, U_3]$; however, $\delta(G) \ge (3/4-\delta)k$, so each vertex of $G$ can have at most $4\delta k$ other missing edges. In particular, in the subgraphs $R'[U_1, U_2]$, $R'[U_3, U_4]$, $B'[U_1,U_3]$, and $B'[U_2, U_4]$, the minimum degree is $\min_j\{|U_j|\} - 4\delta k \ge (1/4-7\delta)k$. 

By Theorem~\ref{hall}, each of these bipartite subgraphs has a matching saturating the smallest part. To see this, consider  without loss of generality $R'[U_1, U_2]$ and assume $|U_1| \le |U_2|$. For $S \subseteq U_1$ with $1 \le |S| \le (1/4-7\delta)k$, $|N(S)| \ge (1/4-7\delta)k \ge |S|$ because any vertex in $S$ has at least $(1/4-7\delta)k$ neighbors in $U_2$. For $S \subseteq U_1$ with $|S| > 7\delta k$, $|N(S)| = |U_2| \ge |S|$ because any vertex in $U_2$ has fewer than $|S|$ non-neighbors in $U_1$. This covers all possibilities, so Hall's condition holds. Moreover, each of these bipartite subgraphs is connected; two vertices in one part share all but at most $14\delta k \le 0.014k$ neighbors in the other part, which has at least $(1/4-3\delta)k \ge 0.247k \ge 2 \cdot 0.014k + 1$ vertices. So each of $R'$ and $B'$ has two connected components, each with a large matching. 

By assumption, there is an edge of the coloring $(R,B)$ that violates the condition in Case~(iii): a blue edge from $U_1 \cup U_2$ to $U_3 \cup U_4$ that is also red, or a red edge from $U_1 \cup U_3$ to $U_2 \cup U_4$ that is also blue. In the first case, this edge connects the two components of $R'$; in the second case, this edge connects the two components of $B'$. In either case, $R'$ or $B'$ becomes connected, and has a matching saturating at least two of $U_1, U_2, U_3, U_4$. We must have $|U_j| \le (1/4+\delta)k$ for all $j$, otherwise the vertices of $U_{5-j}$ would have degree less than $(3/4-\delta)k$. So the matching contains at least $k - 2 (1/4 + \delta)k = (1/2 - 2\delta)k$ edges, and $(1-4\delta)k \ge 0.996k \ge 0.668k \ge (2/3 + \delta)k$ vertices, and Case~(i) of Lemma~\ref{stability-lemma} holds for the coloring $(R,B)$.

Finally, suppose that for every choice of $(R', B')$, Case~(ii) of Lemma~\ref{stability-lemma} holds. We first consider the possibility that for different choices of $(R', B')$ the color in which the sets $S$ have small maximum degree varies. Then there are two choices of $(R', B')$, say $(R_1, B_1)$ and $(R_2, B_2)$, that differ only in the color of one edge, for which sets $S_1, S_2$ exist of order at least $(2/3 - \delta/2)k$ with $\Delta(R_1[S_1]) \le 10\delta k$ and $\Delta(B_2[S_2]) \le 10\delta k$. We have $|S_1 \cap S_2| \ge (1/3 - \delta)k$; let $v$ be a vertex of $S_1 \cap S_2$ such that the two colorings $(R_1, B_1)$ and $(R_2, B_2)$ agree on the edges incident to $v$. (All but at most two vertices of $S_1 \cap S_2$ have this property, since the two colorings only disagree on one edge.) Then $v$ has at most $10\delta k$ edges of $R_1$ to $S_1 \cap S_2$, and at most $10\delta k$ edges of $B_1$ to $S_1 \cap S_2$: altogether $v$ has at most $20\delta k$ neighbors in $S_1 \cap S_2$. Therefore
\[
	\deg(v) \le k - (1/3 - 21\delta)k = (2/3 + 21\delta)k \le 0.687k < 0.749k \le (3/4 - \delta)k,
\]
contradicting our assumption about the minimum degree of $G$.

Therefore Case~(ii) always holds with the sets $S$  inducing small maximum degree in the same color: without loss of generality, red. Choose the coloring $(R', B')$ in which every edge of $E(R) \cap E(B)$ is red. There is a set $S$ of order at least $(2/3-\delta/2)k$ such that $\Delta(R'[S]) \le 10\delta k$; then $\Delta(R[S]) \le 10 \delta k$ as well, and Case~(ii) of Lemma~\ref{stability-lemma} holds for the coloring $(R,B)$.

\paragraph{Acknowledgment.} We thank the referees for their valuable comments.


\begin{thebibliography}{1}

\bibitem{BV} K. Bagga and B. Varma, Bipartite graphs and degree conditions, Graph Theory, Combinatorics, Algorithms and Applications, Proc. 2nd Int. Conf., San Francisco, CA, 1989, 1991, pp. 564–-573.

\bibitem{BKLL2}
J.~Balogh, A.~Kostochka, M.~Lavrov and X.~Liu, Long monochromatic paths and cycles in $2$-edge-colored multipartite graphs,
{\em Moscow J. Comb. Number Theory} 9 (2020), no. 1, 55--100.

\bibitem{BLSSW}
F.~S. Benevides, T.~\L{}uczak, A.~Scott, J.~Skokan and M.~White, Monochromatic cycles in 2-coloured graphs, {\em Combin. Probab. Comput.} 21 (1-2) (2012), 57--87.

\bibitem{BGG}
C.~Berge, Graphs and hypergraphs, Translated from the French by Edward Minieka. Second revised edition. North-Holland Mathematical Library, Vol.~6. North-Holland Publishing Co., Amsterdam-London; American Elsevier Publishing Co., Inc., New York, pp.~ix+528. 1976.

\bibitem{bondy}
J. Bondy, Pancyclic graphs I, {\em J.~Combin.~Theory} 11 (1971), 80--84.

\bibitem{BS}
J. Bondy and M. Simonovits, Cycles of even length in graphs, {\em J.~Combin.~Theory B} 16 (1974), 97--105.

\bibitem{chvatal72}
V.~Chv\'{a}tal,
\newblock On {H}amilton's ideals,
\newblock {\em J.~Combin.~Theory B} 12 (1972), 163--168.








	
\bibitem{FL1}
A.~Figaj and T.~\L uczak, The Ramsey number for a triple of long even cycles, {\em J.~Combin.~Theory B} 97 (2007), 584--596. 


\bibitem{GG1}
L.~Gerencs\' er and A.~Gy\' arf\' as, On Ramsey-type problems, {\em Ann.~Sci.~Budapest.~E\H otv\H os Sect.~Math.}
10 (1967), 167--170.

\bibitem{GL1}
A.~Gy\' arf\' as and J.~Lehel, A Ramsey type problem in directed and bipartite graphs, {\em Periodica Math.~Hung.}
3 (1973), 261--270.



\bibitem{GRSS0}
A.~Gy\' arf\' as, M.~Ruszink\' o,  G.~N.~S\' ark\"{o}zy and E.~Szemer\' edi,
Tripartite Ramsey numbers for paths, {\em J.~Graph Theory} 55 (2007), 164--174.

\bibitem{gyarfas12}
A. Gy\'{a}rf\'{a}s and G.~N. S\'{a}rk\"{o}zy,
\newblock Star versus two stripes {R}amsey numbers and a conjecture of
  {S}chelp,
\newblock {\em Combin. Probab. Comput.} 21 (1-2) (2012), 179--186.

\bibitem{jackson81}
B. Jackson,
\newblock Cycles in bipartite graphs,
\newblock {\em J.~Combin.~Theory B} 30 (3) (1981), 332--342.
{
\bibitem{KSS}
J. Koml\'os, G.~N.~S\' ark\"{o}zy, and E. Szemer\'edi, Blow-Up Lemma. {\em Combinatorica}  17 (1997), 109--123.}

\bibitem{LNS} H. Li, V. Nikiforov and R.~H. Schelp,
\newblock A new type of  {R}amsey--{T}ur\'{a}n  problems,
\newblock {\em Discrete Math.} 272 (2010), 187--196.


\bibitem{L1}		
T.~\L uczak, $R(C_n,C_n,C_n)\leq (4+o(1))n$,  {\em J.~Combin.~Theory B} 75 (1999), 174--187.	


\bibitem{schelp12}
R.~H. Schelp,
\newblock Some {R}amsey--{T}ur\'{a}n type problems and related questions,
\newblock {\em Discrete Math.} 312 (14) (2012), 2158--2161.


\bibitem{Sz}
E.~Szemer\'edi, Regular partitions of graphs, Probl\`emes combinatoires et th\'eorie des graphes (Colloq.~Internat.~CNRS, Univ.~Orsay, Orsay, 1976), pp.~399--401, Colloq.~Internat.~CNRS, 260, CNRS,
Paris, 1978.

\bibitem{W}
M. White, The monochromatic circumference of 2-edge-coloured graphs, {\em J. Graph Theory} 85 (2017), 133--151.







\end{thebibliography}
\end{document}